\documentclass[a4paper]{article}

\usepackage{amsthm}
\usepackage{amsmath}
\usepackage{amssymb}
\usepackage{enumitem}
\usepackage[hang,flushmargin]{footmisc}
\usepackage[margin=25mm]{geometry}
\usepackage[
	pdftitle={Relating ordinary and fully simple maps via monotone Hurwitz numbers},
	pdfauthor={Ga\'{e}tan Borot, S\'{e}verin Charbonnier, Norman Do and Elba Garcia-Failde},
	ocgcolorlinks,
	linkcolor=linkblue,
	citecolor=linkred,
	urlcolor=linkblue]
		{hyperref}
\usepackage{mathpazo}
\usepackage{microtype}
\usepackage{nicefrac}
\usepackage{sectsty}
	\sectionfont{\large}
	\subsectionfont{\normalsize}
\usepackage{setspace}
\usepackage{tikz}
	\usetikzlibrary{positioning, decorations.markings}
\usepackage{titlesec}
	\titlespacing{\section}{0pt}{12pt}{0pt}
	\titlespacing{\subsection}{0pt}{6pt}{0pt}
	\titlespacing{\subsubsection}{0pt}{6pt}{0pt}
\usepackage[nottoc]{tocbibind}
\usepackage{url}
\usepackage[capitalise,noabbrev]{cleveref}
	\crefname{equation}{equation}{equations}
\usepackage{svg}

\theoremstyle{plain}
	\newtheorem{theorem}{Theorem}
	\newtheorem{proposition}[theorem]{Proposition}
	\newtheorem{corollary}[theorem]{Corollary}
	\newtheorem{lemma}[theorem]{Lemma}
	
	\newtheorem*{conjecture*}{Conjecture}
	\numberwithin{theorem}{section}
\theoremstyle{definition}
	\newtheorem{definition}[theorem]{Definition}
	\newtheorem{example}[theorem]{Example}
\theoremstyle{remark}
	\newtheorem*{remark*}{Remark}

\definecolor{linkred}{rgb}{0.75,0,0}
\definecolor{linkblue}{rgb}{0,0,0.75} 
\definecolor{darkblue}{RGB}{0, 0, 139}
\definecolor{darkred}{RGB}{139,0,0}
\definecolor{lightblue}{RGB}{179, 209, 255}
\definecolor{lightred}{rgb}{1,0.85,0.85}

\newcommand\blfootnote[1]{
	\begingroup
	\renewcommand\thefootnote{}\footnote{#1}
	\addtocounter{footnote}{-1}
	\endgroup
}

\setlist{nolistsep}

\newcommand {\h}{\hbar}

\begin{document}

{\large \bfseries Relating ordinary and fully simple maps via monotone Hurwitz numbers}

{\bfseries Ga\"{e}tan Borot\footnote{\label{A1}Max Planck Institut f\"ur Mathematik, Vivatsgasse 7, 53111 Bonn, Germany.}, S\'{e}verin Charbonnier\textsuperscript{\normalfont\ref{A1}}, Norman Do\footnote{School of Mathematical Sciences, Monash University, VIC 3800, Australia.} and Elba Garcia-Failde\textsuperscript{\normalfont\ref{A1},}\footnote{Institut de Physique Th\'eorique-CEA Saclay, Orme des Merisiers, 91191 Gif-sur-Yvette, France.}}

{\em Abstract.} A direct relation between the enumeration of ordinary maps and that of fully simple maps first appeared in the work of the first and last authors. The relation is via monotone Hurwitz numbers and was originally proved using Weingarten calculus for matrix integrals. The goal of this paper is to present two independent proofs that are purely combinatorial and generalise in various directions, such as to the setting of stuffed maps and hypermaps. The main motivation to understand the relation between ordinary and fully simple maps is the fact that it could shed light on fundamental, yet still not well-understood, problems in free probability and topological recursion.
\blfootnote{\par\vspace{-10pt}\emph{Acknowledgements.} G.B., S.C. and E.G.-F. were supported by the Max Planck Gesellschaft. N.D. was supported by the Australian Research Council grant DP180103891. E.G.-F. was also supported by a public grant as part of the Investissement d'avenir project, reference ANR-11-LABX-0056-LMH, LabEx LMH. \par
}

~

\hrule

\setlength{\parskip}{1pt}
\tableofcontents
\setlength{\parskip}{8pt}

~

\hrule

\section{Introduction} \label{sec:introduction}

In this paper, we aim to prove a relation between the enumeration of ordinary maps and the enumeration of fully simple maps that first appeared in the work of the first and last authors~\cite{bor-gar17}. We begin by defining a map, our primary object of study, along with some related notions. In our context, a graph may have loops or multiple edges and we consider it with the topology of a 1-dimensional CW complex.

\begin{definition} \label{def:map}
A {\em map} is a finite graph without isolated vertices embedded into an oriented compact surface. We require the complement of the graph to be a disjoint union of topological disks, which we call {\em faces}.

Define an {\em oriented edge} to be an edge along with a choice of one of its two orientations. We say that an oriented edge is {\em adjacent} to a face if the face lies on its left and {\em incident} to a vertex if it points to the vertex. Maps are endowed with the extra structure of an ordered tuple of distinct oriented edges, such that no two are adjacent to the same face. We refer to these oriented edges as {\em roots}, to the faces adjacent to them as {\em boundary faces}, and to all remaining faces as {\em internal faces}. The number of oriented edges adjacent to a face is called the {\em degree} of the face.
\end{definition}

Two maps are equivalent if there exists an orientation-preserving homeomorphism between their underlying surfaces such that the vertices, oriented edges and faces of the first map are carried bijectively to the vertices, oriented edges and faces of the second, preserving all adjacencies and the tuple of roots.

\begin{figure}[ht!]
\centering
\def\svgwidth{0.38\columnwidth}
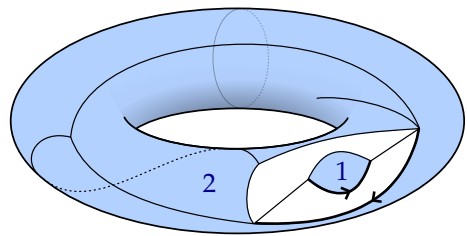
\caption{An example of an ordinary map with two boundary faces of degrees 2 and 11, and two internal faces of degrees 4 and 5.}
\label{fig:map}
\end{figure}

In the final section of the paper, we consider the more general notion of a {\em stuffed map}, which is obtained by relaxing the condition that the complement of the graph is a disjoint union of topological disks.

In general, one says that a map is connected if the underlying topological surface is connected. However, note that our definition of a map does not impose any such condition and indeed, all enumerations considered in this paper include maps that may be disconnected.

The definition of a map allows for different boundary faces to be adjacent along vertices and edges, as well as for a boundary face to be adjacent to itself along vertices and edges. Informally, we call a map fully simple if such behaviour does not arise --- a precise definition follows.

\begin{definition} \label{def:fullysimple}
An oriented edge in a map is a {\em boundary edge} if it is adjacent to a boundary face. A map is {\em fully simple} if each vertex is incident to at most one boundary edge.
\end{definition}

\begin{figure}[ht!]
\centering
\def\svgwidth{0.3\columnwidth}
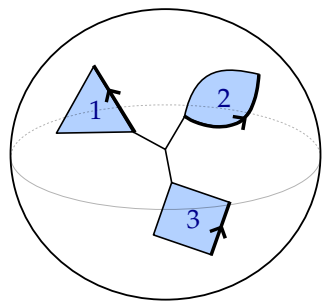
\caption{An example of a fully simple map with three boundary faces of degrees 3, 2 and 4, and an internal face of degree 15.}
\end{figure}

In previous work, the term {\em simple} has been used to refer to maps in which boundary faces are not allowed to be adjacent to themselves along vertices and edges, whereas different boundary faces may be adjacent along vertices and edges~\cite{bor-gar17}.
Throughout, we use the term {\em ordinary} to refer to the class of all maps, so as to emphasise the distinction from the class of fully simple maps. We will be interested primarily in the following enumerations of ordinary and fully simple maps.

\begin{definition} \label{def:mapenumeration}
For positive integers $\mu_1, \mu_2, \ldots, \mu_n$, let $\text{Map}(\mu_1, \mu_2, \ldots, \mu_n)$ denote the weighted enumeration of maps with $n$ boundary faces, such that the degree of boundary face $i$ is $\mu_i$ for $i = 1, 2, \ldots, n$. The weight of a map $M$ is given by
\[
\frac{\h^{-\chi(M)}}{|\mathrm{Aut}\,M|} \, t_1^{f_1(M)} t_2^{f_2(M)} t_3^{f_3(M)} \cdots.
\]
Here, $\chi(M)$ is the Euler characteristic of the underlying surface with the interiors of the boundary faces removed, $f_k(M)$ is the number of internal faces of degree $k$, and $|\mathrm{Aut}\,M|$ is the number of automorphisms. Let $\text{FSMap}(\mu_1, \mu_2, \ldots, \mu_n)$ denote the analogous weighted enumeration restricted to the set of fully simple maps.
\end{definition}

An automorphism of a map is a permutation of the oriented edges arising from an orientation-preserving homeomorphism from the underlying surface to itself that preserves the tuple of roots. Note that if each connected component of $M$ contains at least one boundary face, then $|\mathrm{Aut}\,M| = 1$.

Although there are infinitely many maps with prescribed boundary face degrees, $\text{Map}(\mu_1, \mu_2, \ldots, \mu_n)$ and $\text{FSMap}(\mu_1, \mu_2, \ldots, \mu_n)$ are well-defined elements of $\mathbb{Z}[[\h, \h^{-1}; t_1, t_2, t_3, \ldots]]$. For brevity, our notation makes implicit the dependence on the parameters $\h$ and $t_1, t_2, t_3, \ldots$. From these formal power series, one can extract the number of maps with prescribed boundary face degrees, internal face degrees, and Euler characteristic.

The main result of this paper relates the enumerations of ordinary and fully simple maps via monotone Hurwitz numbers, which we presently describe. Call a sequence $\tau_1, \tau_2, \ldots, \tau_k$ of transpositions in the symmetric group $\mathfrak{S}_d$ {\em strictly monotone} (respectively, {\em weakly monotone}) if $\tau_i = (a_i ~ b_i)$ with $a_i < b_i$ and the sequence $b_1, b_2, \ldots, b_k$ is strictly increasing (respectively, weakly increasing).

\begin{definition} \label{def:monotonehurwitz}
Let $\lambda$ and $\mu$ be partitions of a positive integer $d$ and let $k$ be a non-negative integer. The {\em strictly monotone Hurwitz number} $H_k^<(\lambda; \mu)$ is $\frac{1}{d!}$ times the number of tuples $(\rho_\lambda, \tau_1, \tau_2, \ldots, \tau_k, \rho_\mu)$ of permutations in the symmetric group $\mathfrak{S}_d$ such that
\begin{itemize}
\item $\rho_\lambda$ has cycle type $\lambda$ and $\rho_\mu$ has cycle type $\mu$;
\item $\tau_1, \tau_2, \ldots, \tau_k$ is a strictly monotone sequence of transpositions; and
\item $\rho_\lambda \tau_1 \tau_2 \cdots \tau_k \rho_\mu = \text{id}$.
\end{itemize}
The {\em weakly monotone Hurwitz number} $H^\leq_k(\lambda; \mu)$ is defined analogously, where $\tau_1, \tau_2, \ldots, \tau_k$ is a weakly monotone sequence of transpositions.
\end{definition}

We package these monotone Hurwitz numbers into the generating series
\[
H^<(\lambda; \mu) = \sum_{k\geq 0} H^<_k(\lambda; \mu) \, \h^k \qquad \text{and} \qquad H^\leq(\lambda; \mu) = \sum_{k\geq 0} H^\leq_k(\lambda; \mu) \, \h^k.
\]
Note that $H^<(\lambda; \mu)$ is a polynomial in $\h$, while $H^\leq(\lambda; \mu)$ is in general a formal power series in $\h$. Again, the dependence on $\h$ is implicit in our notation.

For $d$ a non-negative integer, we write $\lambda \vdash d$ to express that $\lambda$ is a partition of $d$. We denote the number of parts of $\lambda$ by $\ell(\lambda)$ and the sum of its elements by $|\lambda|$. Furthermore, we will use the following notation, where $m_j(\lambda)$ is the number of occurrences of the positive integer $j$ in the partition $\lambda$. 
\[
z(\lambda) = \prod_{i=1}^{\ell(\lambda)} \lambda_i \cdot \prod_{j\geq 1} m_j(\lambda)!
\]

The main aim of this paper is to provide two combinatorial proofs of the following result.

\begin{theorem} \label{thm:main}
For any partition $\lambda = (\lambda_1, \lambda_2, \ldots, \lambda_\ell)$ of a positive integer $d$,
\[
\mathrm{Map}(\lambda) = z(\lambda) \sum_{\mu \vdash d} H^<(\lambda; \mu) \, \mathrm{FSMap}(\mu).
\]
\end{theorem}

This result was originally proved using the following techniques from the theory of matrix models~\cite{bor-gar17}. Let $\mathcal{H}(N)$ denote the space of $N\times N$ Hermitian matrices and let $\mathrm{d} \nu$ be a formal measure on it. For a polynomial function $f$ on $\mathcal{H}(N)$, we introduce the notation
\[
\langle f(A) \rangle = \int_{\mathcal{H}(N)} f(A) \,\mathrm{d} \nu(A).
\]
For $\mu_1 + \cdots + \mu_n = d \leq N$, there exists a formal measure $\mathrm{d}\nu$ on $\mathcal{H}(N)$ that is invariant under conjugation by elements of the unitary group $\mathcal{U}(N)$, such that
\[
\text{Map}(\mu_1, \ldots, \mu_n) = \Big\langle \prod_{i = 1}^n \mathrm{Tr} \, A^{\mu_i} \Big\rangle \qquad \text{and} \qquad \text{FSMap}(\mu_1, \ldots, \mu_n) = N^d \Big\langle \prod_{i = 1}^n A_{a[i,1],a[i,2]} \cdots A_{a[i,\mu_i],a[i,1]} \Big\rangle,
\]
under the identification $\h = N^{-1}$. Here, $\big(a[i,1],a[i,2],\ldots,a[i,\mu_i]\big)$ are arbitrary but fixed disjoint cycles in~$\mathfrak{S}_d$. The $\mathcal{U}(N)$-invariance of the measure implies that
\[
\text{FSMap}(\mu_1, \ldots, \mu_n) = N^d \bigg\langle \int_{\mathcal{U}(N)}\; \prod_{i = 1}^n (UAU^{\dagger})_{a[i,1],a[i,2]}\cdots (UAU^{\dagger})_{a[i,\mu_i],a[i,1]}\;\mathrm{d} U \bigg\rangle,
\]
where $\mathrm{d}U$ denotes the Haar measure on $\mathcal{U}(N)$. Weingarten calculus allows one to evaluate the moments of a Haar-distributed random unitary matrix and thus, express the right side of this equation as a linear combination of $\mathrm{Map}(\lambda)$ over partitions $\lambda$~\cite{col03}. The upshot of this calculation is the following relation, which we later show is equivalent to \cref{thm:main}.

\begin{corollary} \label{cor:main}
For any partition $\mu = (\mu_1, \mu_2, \ldots, \mu_n)$ of a positive integer $d$,
\[
\mathrm{FSMap}(\mu) = z(\mu) \sum_{\lambda \vdash d} H^{\leq}(\mu; \lambda) \big|_{\h = -\h} \, \mathrm{Map}(\lambda).
\]
\end{corollary}

In this paper, we provide two combinatorial proofs for \cref{thm:main}. The first one is more direct: we present a simplification algorithm that starts with an ordinary map and produces a fully simple map and a strictly monotone sequence of transpositions. The second one is more geometric: it starts with an ordinary map and produces a fully simple map and a dessin d'enfant, which encodes the pattern of non-simple gluing of the boundary faces. The result then follows from the fact that dessins d'enfant are enumerated by strictly monotone Hurwitz numbers.

\cref{thm:main} was generalised to stuffed maps in the previous work of the first and last authors, following the matrix model approach~\cite{bor-gar17}. The combinatorial approaches presented herein carry over to the setting of stuffed maps, but also generalise to the context of hypermaps, which is perhaps not immediately amenable to the matrix model approach. The essential idea behind these generalisations is the fact that our combinatorial proofs are not sensitive to the behaviour of the internal faces. Since the local structure of the boundary faces in maps, stuffed maps and hypermaps agree, one has the notion of fully simple and analogues of \cref{thm:main} for each case. A discussion of these results will be presented in the final section of the paper.

It is worth remarking here on the genesis of \cref{thm:main} and a possible application. The notion of fully simple maps was introduced in the work of the first and last authors~\cite{bor-gar17}. They show that, analogously to the identification of certain moments with the enumeration of ordinary maps, the free cumulants that arise in free probability theory can be identified with the enumeration of fully simple planar maps. This then provides an elementary tool to work with higher order free cumulants, whose original definition uses intricate objects called partitioned permutations~\cite{col-min-sni-spe07}. Moreover, they propose a combinatorial interpretation of the symplectic invariance property of the topological recursion, which is considered important yet is still not well understood~\cite{eyn-ora07,eyn-ora08,eyn-ora13}. It would be both natural and useful to have a purely combinatorial proof that the enumeration of fully simple maps is governed by the topological recursion, which appears as a conjecture in \cite{bor-gar17}. Since the enumeration of ordinary maps is the prototypical example of a problem governed by the topological recursion, \cref{thm:main} may provide a mechanism to realise such a proof.

The structure of the paper is as follows.
\begin{itemize}
\item In \cref{sec:preliminaries}, we introduce the definitions and conventions for the main players in this paper: namely, ordinary and fully simple maps, monotone Hurwitz numbers, and dessins d'enfant. The permutation model for maps is presented, along with a characterisation of fully simple maps within this framework. We then describe monotone Hurwitz numbers from a representation-theoretic viewpoint and use this to prove the equivalence of \cref{thm:main} and \cref{cor:main}. The final part of this section includes a proof that dessins d'enfant are enumerated by strictly monotone Hurwitz numbers.

\item In \cref{sec:transpositions}, we present the first proof of \cref{thm:main}. The main idea is to start with an ordinary map and to apply a {\em simplification algorithm} that produces a fully simple map and a sequence of transpositions. We then show that the resulting sequence is strictly monotone. Careful accounting of the combinatorial factors and weights involved then allows us to deduce the main theorem of the paper.

\item In \cref{sec:dessins}, we present the second proof of \cref{thm:main}. The main idea is to interpret strictly monotone Hurwitz numbers as an enumeration of dessins d'enfant. The form of the theorem suggests to construct a bijection that takes an ordinary map to a pair comprising a fully simple map and a dessin d'enfant. We present such a construction, as well as its inverse, which allows us to deduce the main theorem of the paper.

\item In \cref{sec:generalisations}, we consider natural generalisations of the notion of fully simple maps and of \cref{thm:main} to stuffed maps and hypermaps.
\end{itemize}

\section{Preliminaries} \label{sec:preliminaries}

In this section we elaborate on our main objects of study. We present the permutation model for maps, give different characterisations for monotone Hurwitz numbers, introduce dessins d'enfant and relate them to strictly monotone Hurwitz numbers.

\subsection{Maps} \label{subsec:maps}

Rather than the topological description of maps provided in \cref{def:map}, we predominantly work with the permutation model for maps. The model is described in the book of Lando and Zvonkin~\cite{lan-zvo04}, although we present it here in notation that is particularly well-suited for our purposes.

One can encode an unrooted map via a triple $(\sigma_0, \sigma_1, \sigma_2)$ of permutations acting on the set $\mathbb{E}$ of oriented edges, in which
\begin{itemize}
\item $\sigma_0$ rotates each oriented edge anticlockwise around the vertex it is incident to;
\item $\sigma_1$ is the fixed point free involution that swaps oriented edges with the same underlying edge; and
\item $\sigma_2$ rotates each oriented edge anticlockwise around the face to its left.
\end{itemize}

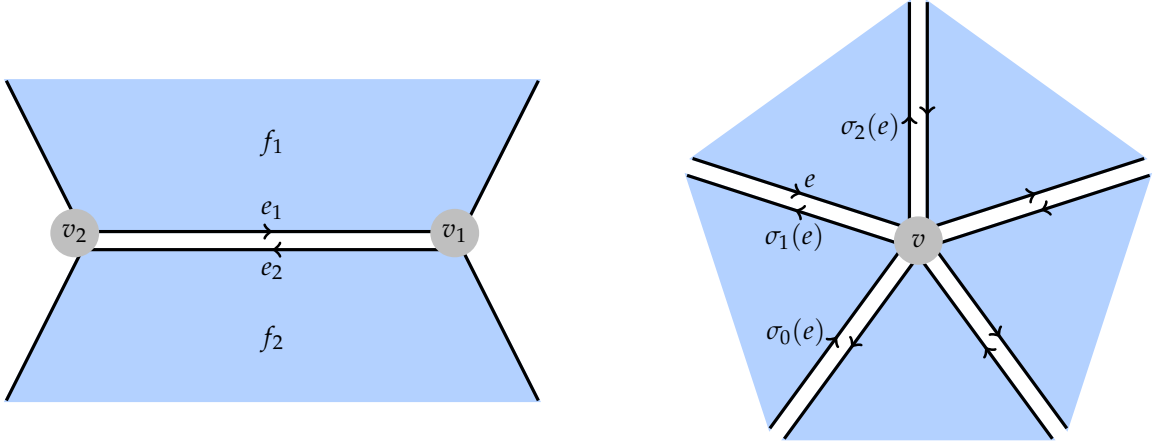
\begin{figure}[ht!]
\centering
\begin{tikzpicture}
\begin{scope}[very thick,decoration={markings, mark=at position 0.5 with {\arrow{>}}}] 
	\filldraw[lightblue] (-1,2.12) -- (0,0.12) -- (5,0.12) -- (6,2.12) -- cycle;
	\filldraw[lightblue] (-1,-2.12) -- (0,-0.12) -- (5,0-.12) -- (6,-2.12) -- cycle;
	\draw[postaction={decorate}] (0,0.12) -- (5,0.12);
	\draw[postaction={decorate}] (5,-0.12) -- (0,-0.12);
	\draw (0,0.12) -- (-1,2.12);
	\draw (5,0.12) -- (6,2.12);
	\draw (0,-0.12) -- (-1,-2.12);
	\draw (5,-0.12) -- (6,-2.12);
	\node at (2.5,0.4) {$e_1$};
	\node at (2.5,-0.4) {$e_2$};
	\node at (2.5,1.3) {$f_1$};
	\node at (2.5,-1.3) {$f_2$};	
	\begin{scope}[shift={(-0.1,0.1)}]
		\filldraw[lightgray] (0,0) circle (0.3);
		\node at (0,0) {$v_2$};
	\end{scope}
	\begin{scope}[shift={(-0.1,0.1)}]
		\filldraw[lightgray] (5,0) circle (0.3);
		\node at (5,0) {$v_1$};
	\end{scope}
\end{scope}
\begin{scope}[shift={(11,0)},very thick,decoration={markings, mark=at position 0.5 with {\arrow{>}}}] 
	\begin{scope}[rotate=18]
		\filldraw[lightblue,shift={(36:0.2)}] (0:3) -- (0,0) -- (72:3) -- cycle;
		\draw[postaction={decorate},shift={(36:0.2)}] (0,0) -- (0:3);
		\draw[postaction={decorate},shift={(36:0.2)}] (72:3) -- (0,0);
	\end{scope}
	\begin{scope}[rotate=90]
		\filldraw[lightblue,shift={(36:0.2)}] (0:3) -- (0,0) -- (72:3) -- cycle;
		\draw[postaction={decorate},shift={(36:0.2)}] (0,0) -- (0:3);
		\draw[postaction={decorate},shift={(36:0.2)}] (72:3) -- (0,0);
	\end{scope}
	\begin{scope}[rotate=162]
		\filldraw[lightblue,shift={(36:0.2)}] (0:3) -- (0,0) -- (72:3) -- cycle;
		\draw[postaction={decorate},shift={(36:0.2)}] (0,0) -- (0:3);
		\draw[postaction={decorate},shift={(36:0.2)}] (72:3) -- (0,0);
	\end{scope}
	\begin{scope}[rotate=234]
		\filldraw[lightblue,shift={(36:0.2)}] (0:3) -- (0,0) -- (72:3) -- cycle;
		\draw[postaction={decorate},shift={(36:0.2)}] (0,0) -- (0:3);
		\draw[postaction={decorate},shift={(36:0.2)}] (72:3) -- (0,0);
	\end{scope}
	\begin{scope}[rotate=306]
		\filldraw[lightblue,shift={(36:0.2)}] (0:3) -- (0,0) -- (72:3) -- cycle;
		\draw[postaction={decorate},shift={(36:0.2)}] (0,0) -- (0:3);
		\draw[postaction={decorate},shift={(36:0.2)}] (72:3) -- (0,0);
	\end{scope}
	\filldraw[lightgray] (0,0) circle (0.3);
	\node at (0,0) {$v$};
	\node[above,shift={(0,0.1)}] at (162:1.5) {$e$};
	\node[left,shift={(-0.1,0)}] at (90:1.5) {$\sigma_2(e)$};
	\node[below,shift={(-0.2,-0.1)}] at (162:1.5) {$\sigma_1(e)$};
	\node[left,shift={(-0.2,0)}] at (234:1.5) {$\sigma_0(e)$};
\end{scope}
\end{tikzpicture}
\caption{The left diagram depicts the local structure of an edge in a map. The oriented edges $e_1$ and $e_2$ are indicated by the arrows. With our conventions, $e_i$ is {\em adjacent} to face $f_i$ and {\em incident} to vertex $v_i$ for $i = 1, 2$. The right diagram depicts the local structure of a vertex in a map, including the action of the permutations $\sigma_0, \sigma_1, \sigma_2$ on an oriented edge $e$.}
\label{fig:vertex}
\end{figure}

It follows that $\sigma_0 \sigma_1 \sigma_2 = \mathrm{id}$, where we adopt the convention of multiplying permutations from right to left. Thus, one obtains the following result.

\begin{lemma}
A map can be encoded by a triple $(\sigma_0, \sigma_1, \sigma_2)$ of permutations in $\mathfrak{S}(\mathbb{E})$ and a tuple $R \in \mathbb{E}^n$ such that
\begin{itemize}
\item $\sigma_1$ is a fixed point free involution;
\item $\sigma_0 \sigma_1 \sigma_2 = \mathrm{id}$; and
\item no two elements of $R$ lie in the same cycle of $\sigma_2$.
\end{itemize}
The data $(\sigma_0, \sigma_1, \sigma_2; R)$ and $(\widetilde{\sigma_0}, \widetilde{\sigma_1}, \widetilde{\sigma_2}; \widetilde{R})$ define equivalent maps if and only if there exists a bijection $\phi: \mathbb{E} \to \widetilde{\mathbb{E}}$ that sends $R$ to $\widetilde{R}$ and satisfies $\widetilde{\sigma}_i = \phi \sigma_i \phi^{-1}$ for $i \in \{0, 1, 2\}$.
\end{lemma}

This permutation model admits the following characterisation of fully simple maps. Suppose that a map is given by the data $(\sigma_0, \sigma_1, \sigma_2; R)$. Define the set $B \subseteq \mathbb{E}$ to be the union of the $\sigma_2$-orbits of the elements of $R$ and observe that this naturally corresponds to the set of boundary edges. Then the map is fully simple if and only if the elements of $B$ lie in different $\sigma_0$-orbits.

Let us describe the characterisation of fully simple maps in a slightly different way, using a notation that will subsequently be useful. Denote by $\sigma_0^\partial \in \mathfrak{S}(B)$ the permutation obtained by expressing $\sigma_0 \in \mathfrak{S}(\mathbb{E})$ as a union of disjoint cycles and deleting those elements that do not lie in $B$. If $e \in B$ is an oriented edge incident to the vertex $v$, then $\sigma_0^\partial(e)$ is the next oriented edge in $B$ incident to $v$ that is encountered when turning anticlockwise around $v$. Then a map is fully simple if and only if the permutation $\sigma_0^\partial$ is the identity permutation.

\subsection{Monotone Hurwitz numbers} \label{subsec:monotonehurwitz}

\cref{def:monotonehurwitz} describes strictly and weakly monotone Hurwitz numbers as the enumeration of certain factorisations in the symmetric group. Such problems are often amenable to calculation via the representation theory of the symmetric group. For a positive integer $d$, consider the centre $Z \mathbb{Q}[\mathfrak{S}_d]$ of the symmetric group algebra. As a vector space, it has a basis formed by the conjugacy classes $C_\lambda$, defined to be the sum of the permutations whose cycle type is given by the partition $\lambda$ of $d$.

The representation theory of the symmetric group may be understood through the Jucys--Murphy elements $J_m = \sum_{\ell=1}^{m-1} (\ell ~ m) \in \mathbb{Q}[\mathfrak{S}_d]$ for $m = 2, 3, \ldots, d$~\cite{juc74,mur81}. The Jucys--Murphy elements commute and it follows that any symmetric polynomial of $J_2,J_3, \ldots, J_d$ is an element of $Z\mathbb{Q}[\mathfrak{S}_d]$.

The following equations demonstrate that the monotone Hurwitz numbers can be expressed in terms of the centre of the symmetric group algebra as well as in terms of characters of the symmetric group.
\begin{align} 
H^<(\lambda; \mu) &= \sum_{k\geq 0} H^<_k(\lambda; \mu) \, \h^k = \frac{1}{d!} [\mathrm{id}] \bigg(C_{\lambda} C_{\mu} \prod_{m = 2}^d (1 + \h J_{m})\bigg) = \sum_{\rho \vdash d} \frac{\chi_\rho(\lambda)\chi_\rho(\mu)}{z(\lambda) z(\mu)} \prod_{\Box \in \rho} (1 + c(\Box) \h) \label{eq:smhurwitz} \\
H^\leq(\lambda; \mu) &= \sum_{k\geq 0} H^\leq_k(\lambda; \mu) \, \h^k = \frac{1}{d!} [\mathrm{id}] \bigg(C_{\lambda} C_{\mu} \prod_{m = 2}^d \frac{1}{1 - \h J_{m}}\bigg) = \sum_{\rho \vdash d} \frac{\chi_\rho(\lambda) \chi_\rho(\mu)}{z(\lambda) z(\mu)} \prod_{\Box \in \rho} \frac{1}{1 - c(\Box) \h} \label{eq:wmhurwitz}
\end{align}
The notation $\chi_\rho(\lambda)$ refers to the symmetric group character indexed by $\rho$ evaluated on a permutation of cycle type~$\lambda$. The final product in each line is over the boxes of the Young diagram of the partition $\rho$. The notation $c(\Box)$ refers to the content of the box, which is defined to be $j-i$ for a box in the $i$th row from the top and the $j$th column from the left.

In both \cref{eq:smhurwitz,eq:wmhurwitz}, the first equality is the definition of the monotone Hurwitz number generating series. The second equality arises from expanding the product of conjugacy classes with the symmetric polynomials of the Jucys--Murphy elements in $Z \mathbb{Q}[\mathfrak{S}_d]$ and collecting the coefficient of the identity. The third equality is obtained by converting the conjugacy classes into the basis of orthogonal idempotents in $Z \mathbb{Q}[\mathfrak{S}_d]$ and invoking the Jucys correspondence~\cite{juc74}.

\begin{proof}[Proof of \cref{cor:main} from \cref{thm:main}]
Recall that \cref{thm:main} and \cref{cor:main} respectively state that
\[
\mathrm{Map}(\lambda) = \sum_{\mu \vdash d}z(\lambda) H^<(\lambda; \mu) \, \mathrm{FSMap}(\mu) \qquad \text{and} \qquad \mathrm{FSMap}(\mu) = \sum_{\lambda \vdash d}z(\mu) \left. H^\leq(\mu; \lambda) \right|_{\h = -\h} \, \mathrm{Map}(\lambda).
\]
These equations provide the transition matrices that convert from ordinary to fully simple map enumerations and vice versa. The equivalence of these two statements is a consequence of the fact that these transition matrices are inverses of each other. To prove this, it is sufficient to check that
\[
\sum_{\rho \vdash d} \Big( z(\lambda) \, H^<(\lambda; \rho) \Big) \cdot \Big( z(\rho) \, H^\leq(\rho; \mu) \big|_{\h = -\h} \Big) = \delta_{\lambda,\mu}.
\] 
The check is a straightforward consequence of applying \cref{eq:smhurwitz,eq:wmhurwitz} and the orthogonality of characters.
\end{proof}

\subsection{Dessins d'enfant}\label{subsec:dessins}

A dessin d'enfant is often described in the literature as a map whose vertices are bicoloured in such a way that each edge is adjacent to one vertex of each colour~\cite{lan-zvo04}. However, we will adopt the dual picture, which aligns with \cref{def:map} and is geometrically well-suited to our purposes.

\begin{definition}
A {\em dessin d'enfant} is a map in which each edge is adjacent to one boundary face and one internal face. In this context, we refer to the boundary faces as {\em blue faces} and the internal faces as {\em red faces}. Two dessins d'enfant are equivalent if the corresponding maps are equivalent.
\end{definition}

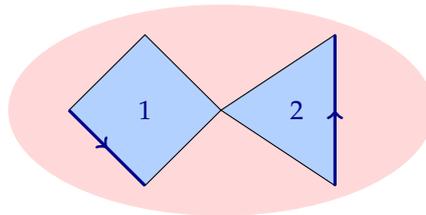
\begin{figure}[ht!]
\centering
\begin{tikzpicture}
\fill[fill = lightred] (1,1) ellipse (2.8 and 1.4);
\draw[fill = lightblue] (0,0) -- (1,1) -- (0,2) -- (-1,1) -- cycle;
\draw[fill = lightblue] (1,1) -- (2.5,0) -- (2.5,2) -- cycle;
\node[darkblue] at (0,1) {1};
\node[darkblue] at (2,1) {2};
\begin{scope}[very thick,decoration={markings,mark=at position 0.5 with {\arrow{>}}}] 
	\draw[darkblue, postaction={decorate}] (-1,1) -- (0,0);
	\draw[darkblue, postaction={decorate}] (2.5,0) -- (2.5,2);
\end{scope}
\end{tikzpicture}
\caption{An example of a dessin d'enfant on the sphere with two blue faces and one red face.}
\label{fig:dessin}
\end{figure}

One can encode a dessin d'enfant via a triple $(\tau_r, \tau_b, \tau_v)$ of permutations acting on the set $E$ of (unoriented) edges, in which
\begin{itemize}
\item $\tau_r$ rotates each edge anticlockwise around the adjacent red face; 
\item $\tau_b$ rotates each edge anticlockwise around the adjacent blue face; and
\item $\tau_v$ rotates each edge anticlockwise by two edges around the vertex to which it points.
\end{itemize}
For this last point, we assign an orientation to the edges of a dessin d'enfant in which each edge is oriented to have a blue face on its left and a red face on its right.

It follows that $\tau_r \tau_b \tau_v = \mathrm{id}$, where we adopt the convention of multiplying permutations from right to left. Thus, one obtains the following result.

\begin{lemma}
A dessin d'enfant can be encoded by a triple $(\tau_r, \tau_b, \tau_v)$ of permutations in $\mathfrak{S}(E)$ and a tuple $R \in E^n$ such that 
\begin{itemize}
\item $\tau_r \tau_b \tau_v = \mathrm{id}$; and
\item each cycle of $\tau_b$ contains exactly one element of $R$.
\end{itemize}
The data $(\tau_r, \tau_b, \tau_v; R)$ and $(\widetilde{\tau}_r, \widetilde{\tau}_b, \widetilde{\tau}_v; \widetilde{R})$ define equivalent dessins d'enfant if and only if there exists a bijection $\phi: E \to \widetilde{E}$ that sends $R$ to $\widetilde{R}$ and satisfies $\widetilde{\tau}_i = \phi \tau_i \phi^{-1}$ for $i \in \{r, b, v\}$.
\end{lemma}

\begin{definition}
Let $\lambda = (\lambda_1, \lambda_2, \ldots, \lambda_\ell)$ and $\mu = (\mu_1, \mu_2, \ldots, \mu_n)$ be partitions of a non-negative integer~$d$ and let $k$ be an integer. Define $D_k(\lambda; \mu)$ to be the number of (possibly disconnected) dessins d'enfant such that 
\begin{itemize}
\item the blue face $i$ has degree $\lambda_i$ for $i = 1, 2, \ldots, \ell$;
\item the red faces have degrees $\mu_1, \mu_2, \ldots, \mu_n$ in some arbitrary order; and
\item the number of edges is $k$ more than the number of vertices.
\end{itemize}
\end{definition}

Up to a simple combinatorial factor, the strictly monotone Hurwitz number $H^<_k(\lambda; \mu)$ and the dessin d'enfant enumeration $D_k(\lambda; \mu)$ agree. The crux of the argument is the following elementary result, which connects the two proofs of \cref{thm:main} presented in the following sections.

\begin{lemma} \label{lem:dessins}
Each permutation in $\mathfrak{S}_d$ can be uniquely expressed as the product of a strictly monotone sequence of transpositions. Moreover, if the permutation has cycle type $\nu$, then the number of transpositions is $d - \ell(\nu)$.
\end{lemma}

\begin{proof}
A cycle $(a_1 ~ a_2 ~ \cdots ~ a_m)$ in which $a_m = \max(a_1, a_2, \ldots, a_m)$ may be expressed as the product $(a_1 ~ a_2 ~ \cdots ~ a_{m-1}) \circ (a_{m-1} ~ a_m)$. Iterating the process on the smaller cycle $(a_1 ~ a_2 ~ \cdots ~ a_{m-1})$ results in an expression for $(a_1 ~ a_2 ~ \cdots ~ a_m)$ as a product of a strictly monotone sequence of $m-1$ transpositions.

For an arbitrary permutation $\rho \in \mathfrak{S}_d$ of cycle type $\nu$, one may perform the above procedure to each cycle to obtain an expression
\[
\rho = (a_1 ~ b_1) \circ (a_2 ~ b_2) \circ \cdots \circ (a_k ~ b_k),
\]
where $a_i < b_i$ and $b_1, b_2, \ldots, b_k$ are pairwise distinct. Now commute the transpositions to ensure that the resulting sequence is strictly monotone. If we have two consecutive transpositions $(a ~ b) \circ (c ~ d)$ with $a, b, c, d$ pairwise distinct, then they commute and we have $(a ~ b) \circ (c ~ d) = (c ~ d) \circ (a ~ b)$. The only other case that arises is if we have two consecutive transpositions $(a ~ c) \circ (a ~ b)$ with $a < b < c$, in which case we have $(a ~ c) \circ (a ~ b) = (a ~ b) \circ (b ~ c)$. Repeatedly applying these operations results in an expression for $\rho$ as the product of a strictly monotone sequence of $d - \ell(\nu)$ transpositions.

To see why this expression is unique, we simply show that the number of strictly monotone sequences of transpositions in $\mathfrak{S}_d$ is equal to the number of permutations in $\mathfrak{S}_d$. One way to see this is to consider sequences
\[
(a_2 ~ 2), (a_3 ~ 3), (a_4 ~ 4), \ldots, (a_d ~ d),
\]
where $1 \leq a_k \leq k$. It is clear that the number of such sequences is $d!$ and one obtains all possible strictly monotone sequences of transpositions by redacting any occurrences of $(i ~ i)$ for some integer $i$.
\end{proof}

Consider a dessin d'enfant $(\tau_r, \tau_b, \tau_v; R)$ in which the cycle types of $\tau_b$ and $\tau_r$ are $\lambda$ and $\mu$, respectively. The previous proposition allows us to write $\tau_v = \tau_1 \tau_2 \cdots \tau_k$ for a unique strictly monotone sequence of transpositions $\tau_1, \tau_2, \ldots, \tau_k$. It follows that
\[
\tau_b \tau_1 \tau_2 \cdots \tau_k \tau_r = \text{id},
\]
so we obtain a tuple $(\tau_b, \tau_1, \tau_2, \ldots, \tau_k, \tau_r)$ that contributes to the strictly monotone Hurwitz number $H^<_k(\lambda; \mu)$. Recall that the enumeration of dessins d'enfant required in addition a choice of the tuple of roots. The number of such choices is simply
\[
z(\lambda) = \prod_{i=1}^{\ell(\lambda)} \lambda_i \cdot \prod_{j\geq 1} m_j(\lambda)!.
\]
The first product accounts for the number of ways to choose a root within each cycle, while the second product accounts for the number of ways to order these so that root $r_j$ comes from a cycle of length $\lambda_j$. Thus, we obtain the following relation.

\begin{proposition} \label{prop:dessins}
The strictly monotone Hurwitz numbers and the dessin d'enfant enumeration are related by
\[
D_k(\lambda; \mu) = z(\lambda) \, H_k^<(\lambda;\mu).
\]
\end{proposition}

In analogy with the monotone Hurwitz numbers, we collect the $D_k(\lambda; \mu)$ for varying $k$ together in the generating series
\[
D(\lambda; \mu) = \sum_{k \geq 0} D_k(\lambda; \mu) \, \h^k.
\]
The previous proposition can then be expressed as
\[
D(\lambda; \mu) = z(\lambda) \, H^<(\lambda;\mu).
\]

\section{Proof 1: Monotone transpositions} \label{sec:transpositions}

In this section, we present an algorithm that turns an ordinary map into a fully simple map. The algorithm systematically traverses the set of boundary edges and performs a ``simplification'' there, if possible.

\subsection{Simplification algorithm} \label{subsec:algorithm}

Throughout the section, we use the following terminology.

\begin{definition}
A vertex in a map is called \emph{fully simple} if at most one boundary edge is incident to it. A boundary face in a map is called \emph{fully simple} if every vertex incident to it is fully simple. Thus, a map is \emph{fully simple} if and only if all of its vertices are fully simple or equivalently, if and only if all of its boundary faces are fully simple.
\end{definition}

Now let us start with an ordinary map $M$ with $\ell$ boundary faces of respective degrees $\lambda_1, \lambda_2, \ldots, \lambda_\ell$. We assign to the oriented edges adjacent to boundary face $i$ the labels
\[
(i, 1), (i, 2), (i, 3), \ldots, (i, \lambda_i),
\]
where $(i,1)$ denotes the root and the remaining labels are assigned in an anticlockwise manner around the boundary face. With this convention, one can write $\sigma_2(i,j)=(i,j+1)$, where the second entry is considered modulo $\lambda_i$. These labels allow us to equip the set $B$ of boundary edges with the lexicographical order. From $M$, we construct a fully simple map $M^s$ via the following algorithm.

We start at the root $(1,1)$ of boundary face 1 in $M$. Our algorithm traverses the set $B$ of boundary edges in lexicographical order. At each step of the algorithm, the permutations $\sigma_0$ and $\sigma_2$ may change, while $\sigma_1$ remains unchanged throughout. Suppose that we are at the boundary edge $(p,q)$ and that it is incident to vertex $v$. Then the following two possibilities arise.

\begin{itemize}
\item If the vertex $v$ is fully simple, then we leave the permutations $\sigma_0$ and $\sigma_2$ unchanged.

\item Otherwise, there are at least two boundary edges incident to $v$, including $(p,q)$. Let us write $(p',q') = \sigma_0^\partial(p,q)$ and observe that since $v$ is not fully simple, we must have $(p',q') \neq (p,q)$ --- see \cref{fig:vertexalgo1}. We change the permutation $\sigma_0$ into $\tilde{\sigma}_0$ by composing it with a transposition thus.
\begin{equation} \label{eq:tit}
\tilde{\sigma}_0 = ((p,q); (p',q')) \circ \sigma_0
\end{equation}
To preserve the relation $\sigma_0 \sigma_1 \sigma_2 = \mathrm{id}$, we change the permutation $\sigma_2$ into $\tilde{\sigma}_2$ in the following way.
\[
\tilde{\sigma}_2= \sigma_2 \circ ((p,q);(p',q'))
\]
This step of the algorithm splits the vertex $v$ into two vertices $v_1$ and $v_2$, as shown in \cref{fig:vertexalgo1}. The oriented edge $(p,q)$ is now incident to $v_1$, which is necessarily fully simple. The oriented edge $(p',q')$ is now incident to $v_2$, which might not be fully simple. At the end of each step of the algorithm, we update $\sigma_0$ to be $\tilde{\sigma}_0$ and $\sigma_2$ to be $\tilde{\sigma}_2$.
\end{itemize}

\begin{figure}[ht!]
\centering
\begin{tikzpicture}
\begin{scope}[very thick,decoration={markings, mark=at position 0.5 with {\arrow{>}}}] 
	\begin{scope}[shift={(-135:0.32)}]
		\filldraw[lightblue] (-160:3) -- (20:0) -- (-110:3) -- cycle;
		\draw[postaction={decorate}] (-110:3) -- (0,0);
		\draw[postaction={decorate}] (0,0) -- (-160:3);
	\end{scope}
	\begin{scope}[shift={(45:0.32)}]
		\filldraw[lightblue] (20:3) -- (20:0) -- (70:3) -- cycle;
		\draw[postaction={decorate}] (70:3) -- (0,0);
		\draw[postaction={decorate}] (0,0) -- (20:3);
	\end{scope}
	\begin{scope}[shift={(135:0.16)}]
		\draw[dashed, thin] (70:2) arc (70:200:2);
		\draw[postaction={decorate}] (-160:3) -- (0,0);
		\draw[postaction={decorate}] (0,0) -- (70:3);
	\end{scope}
	\begin{scope}[shift={(-45:0.16)}]
		\draw[dashed, thin] (-110:2) arc (-110:20:2);
		\draw[postaction={decorate}] (20:3) -- (0,0);
		\draw[postaction={decorate}] (0,0) -- (-110:3);
		\node[align=center] at (-45:1) {no other \\ boundary};
	\end{scope}	
	\filldraw[lightgray] (0,0) circle (0.3);
	\node at (0,0) {$v$};
	\node[anchor=south east] at (-110:3) {$(p,q)$};
	\node[anchor=north west] at (70:3) {$(p',q')$};
	\node[anchor=north west] at (-160:3) {$\sigma_2(p,q)$};
	\node[anchor=south east] at (20:3) {$\sigma_2(p',q')$};
\end{scope}
\begin{scope}
	[shift={(8,0)},very thick,decoration={markings, mark=at position 0.5 with {\arrow{>}}}] 
	\begin{scope}[shift={(135:0.45)}]
		\filldraw[lightblue] (-160:3) -- (0,0) -- (70:3) -- (20:3) -- (0.4,-0.4) -- (-110:3) -- cycle;
		\draw[postaction={decorate}] (0,0) -- (-160:3);
		\draw[postaction={decorate}] (70:3) -- (0,0);
		\node[anchor=north west] at (70:3) {$(p',q')$};
	\end{scope}
	\begin{scope}[shift={(135:0.75)}]
		\draw[dashed, thin] (70:2) arc (70:200:2);
		\draw (-160:3) -- (0,0);
		\draw (0,0) -- (70:3);
		\begin{scope}[shift={(0.07,-0.07)}]
			\filldraw[lightgray] (0,0) circle (0.3);
			\node at (0,0) {$v_2$};
		\end{scope}
	\end{scope}
	\begin{scope}[shift={(-45:0.45)}]
		\filldraw[lightblue] (-110:3) -- (0,0) -- (20:3) -- (70:3) -- (-0.4,0.4) -- (-160:3) -- cycle;
		\draw[postaction={decorate}] (0,0) -- (20:3);
		\draw[postaction={decorate}] (-110:3) -- (0,0);
		\node[anchor=south east] at (20:3) {$\sigma_2(p',q')$};
		\node[anchor=south east] at (-110:3) {$(p,q)$};
	\end{scope}
	\begin{scope}[shift={(-45:0.75)}]
		\draw[dashed, thin] (-110:2) arc (-110:20:2);
		\draw (20:3) -- (0,0);
		\draw (0,0) -- (-110:3);
		\node[align=center] at (-45:1) {no other \\ boundary};
		\begin{scope}[shift={(-0.07,0.07)}]
			\filldraw[lightgray] (0,0) circle (0.3);
			\node at (0,0) {$v_1$};
		\end{scope}
	\end{scope}
	\node[anchor=north west,shift={(135:0.4)}] at (-160:3) {$\sigma_2(p,q)$};
\end{scope}
\end{tikzpicture}
\caption{The diagrams depict the local structure before and after the simplification algorithm is applied to the boundary edge $(p,q)$, which is incident to a vertex $v$ that is not fully simple. The arrows represent oriented edges while the blue domains represent boundary faces. We turn anticlockwise around $v$ and seek the first boundary edge $(p',q') = \sigma_0^\partial(p,q)$. The permutation $\sigma_0$ is then updated to $\tilde{\sigma}_0=((p,q);(p',q')) \circ \sigma_0$ and the permutation $\sigma_2$ to $\tilde{\sigma}_2= \sigma_2 \circ ((p,q);(p',q'))$. This operation has the effect of splitting the vertex $v$ into two vertices $v_1$ and $v_2$.} \label{fig:vertexalgo1}
\end{figure}
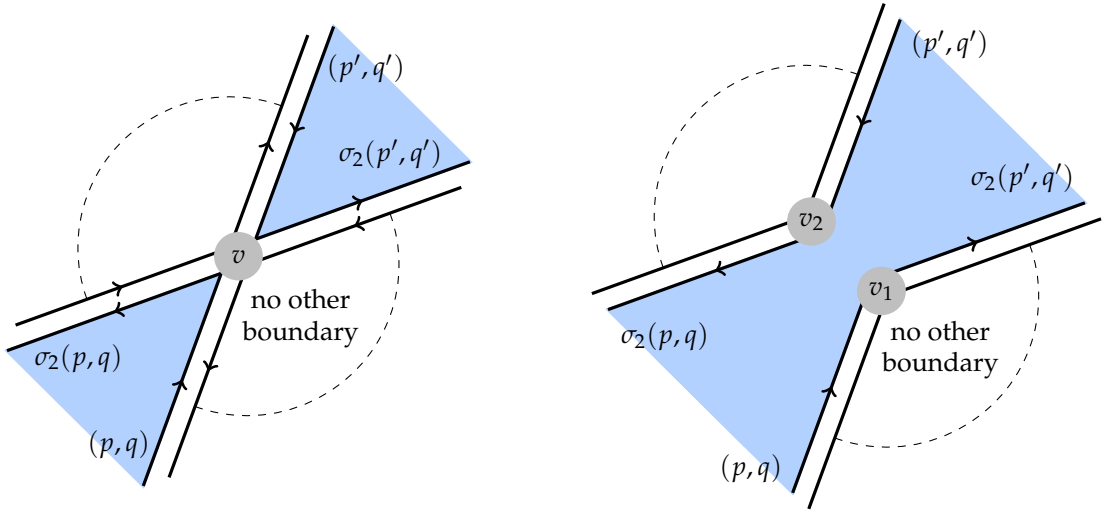

Although we distinguished two cases here, one may consider the first to be a special case of the second; for if $v$ is already fully simple, then one has $\sigma_0^\partial(p,q) = (p,q)$. So the ``transposition'' arising in \cref{eq:tit} would be $((p,q); (p,q))$, which one may interpret as the identity permutation. So composing with the identity permutation is consistent with leaving the permutations $\sigma_0$ and $\sigma_2$ unchanged in the first case.

At the end of each step, we move to the next boundary edge according to the lexicographical order defined above. Once all of the oriented edges adjacent to a boundary face have been traversed, that boundary face is then fully simple. Once all of the boundary edges in $B$ have been traversed, then the resulting map is fully simple, although it remains to assign the roots. We do this by declaring the boundary faces to be such that $B$ remains the set of boundary edges. We declare the tuple of roots to be lexicographically minimal, such that there is one root adjacent to each such boundary face. Thus, we conclude the algorithm with a fully simple map, which we denote by $M^s$.

\subsection{Monotonicity}

Let $\tau_1, \tau_2, \ldots, \tau_k$ be the transpositions appearing in \cref{eq:tit} during the simplification algorithm, in the order that they arise. The permutations representing the ordinary map $M$ are $\sigma_0, \sigma_1, \sigma_2$, and we denote the permutations representing the fully simple map $M^s$ by $\sigma_0^s, \sigma_1^s, \sigma_2^s$. The latter are obtained from the former via
\begin{equation} \label{eq:bij}
\sigma_0^s= \tau_k\cdots\tau_2\tau_1\sigma_0, \qquad \sigma_1^s=\sigma_1,\qquad \sigma_2^s= \sigma_2\tau_1\tau_2\cdots\tau_k.
\end{equation}

The algorithm implies a certain monotonicity condition on the sequence of transpositions $\tau_1, \tau_2, \ldots, \tau_k$. We write $\tau_i = ((p_i,q_i); (p_i',q_i'))$ for $i=1, 2, \ldots,k$, where $(p_i',q_i') = \sigma_0^\partial(p_i,q_i)$ and we adopt the labels of boundary edges described in \cref{subsec:algorithm}. By construction, we know that $(p_i, q_i)$ is smaller than $(p_i',q_i')$ with respect to the lexicographical order. Otherwise, the boundary edge $(p_i',q_i')$ would have been visited in a previous step of the algorithm and the vertex it is incident to would have already been made fully simple. Furthermore, we know that the sequence $(p_1,q_1), (p_2,q_2), \ldots, (p_k,q_k)$ is strictly increasing with respect to the lexicographical order. This is because the algorithm visits the boundary edges in that order. Thus, we have deduced the following.

\begin{lemma} \label{prop:MonotonousTranspo}
For any ordinary map $M$, the sequence $\tau_1, \tau_2, \ldots, \tau_k$ of transpositions arising from the simplification algorithm satisfies the following monotonicity property with respect to the lexicographical order: the smaller elements transposed by $\tau_1, \tau_2, \ldots, \tau_k$ form a strictly increasing sequence.
\end{lemma}

Note that in \cref{def:monotonehurwitz}, the definition of monotone Hurwitz numbers requires a sequence of transpositions in which the larger elements form a strictly increasing sequence. However, we claim that the enumeration of such sequences is not sensitive to whether one takes the smaller or larger elements. Indeed, one can obtain one from the other by reversing the sequence of transpositions and reversing the ordering imposed on the elements.

\subsection{Conclusion}

The previous discussions on the simplification algorithm and monotonicity now allow us to deduce our main result.

\begin{proof}[Proof of \cref{thm:main}]
To any ordinary map $M$ described by the permutations $\sigma_0, \sigma_1, \sigma_2$, the simplification algorithm above associates a fully simple map $M^s$ described by the permutations $\sigma_0^s, \sigma_1^s, \sigma_2^s$, as well as a strictly increasing sequence of transpositions $\tau_1, \tau_2, \ldots, \tau_k$. All of these permutations are related by \cref{eq:bij}.

Observe that the simplification algorithm changes neither the number of internal faces nor their degrees. It also leaves the number of edges invariant but creates $k$ new vertices, as shown in \cref{fig:vertexalgo1}. Since we calculate Euler characteristics after removing the interiors of boundary faces, we find that
\[
\chi(M^s) = \chi(M) + k.
\]
It follows that the weight attached to $M$ is $\h^k$ multiplied by the weight attached to $M^s$.

By inverting \cref{eq:bij}, one deduces that the correspondence between maps and fully simple maps with a strictly increasing sequence of transpositions is bijective and weight-preserving. Therefore, we have
\begin{equation} \label{eq:proof1}
\mathrm{Map}(\lambda) = \sum_{\mu \vdash |\lambda|} \bigg(\sum_{k \geq 0} \widetilde{H}_k^{<}(\lambda; \mu) \, \h^k \bigg) \, \mathrm{FSMap}(\mu),
\end{equation}
where $\widetilde{H}_k^<(\lambda; \mu)$ is the number of strictly increasing sequences $\tau_1, \tau_2, \ldots, \tau_k$ of transpositions such that
\[
\sigma_2^s = \sigma_2 \circ \tau_1 \circ \tau_2 \circ \cdots \circ \tau_k \in C_\mu.
\]
To compute $\widetilde{H}_k^<(\lambda; \mu)$, we emphasise that the permutation $\sigma_2$ is a fixed element of $C_\lambda$, while $\sigma_2^s$ is allowed to be an arbitrary element of $C_\mu$. From the discussion in \cref{subsec:monotonehurwitz}, we know that starting from another fixed permutation $\tilde{\sigma}_2 \in C_{\lambda}$ will produce the same number. After comparing with \cref{def:monotonehurwitz}, we deduce that
\[
\widetilde{H}_k^<(\lambda; \mu) = \frac{d!}{|C_{\lambda}|} H_k^<(\lambda; \mu) = z(\lambda) H_k^<(\lambda; \mu),
\]
and combining with \cref{eq:proof1} yields the desired result.
\end{proof}

\section{Proof 2: Dessins d'enfant} \label{sec:dessins}

In this section we construct a bijection between ordinary maps and pairs comprising a fully simple map and a dessin d'enfant.

\subsection{Intuition}

In \cref{subsec:dessins}, we observed that strictly monotone Hurwitz numbers are naturally related to dessins d'enfant. In particular, one may invoke \cref{prop:dessins} to equivalently state \cref{thm:main} as
\[
\mathrm{Map}(\lambda) = \sum_{\mu \vdash d} D(\lambda; \mu) \, \mathrm{FSMap}(\mu).
\]
This particular form of the theorem suggests a natural combinatorial proof by constructing a function
\[
\text{map} \quad \longmapsto \quad (\text{fully simple map}, \text{dessin d'enfant})
\]
that takes an ordinary map and returns a pair comprising a fully simple map and a dessin d'enfant. Moreover, we would like the blue face degrees of the dessin d'enfant to match the boundary face degrees of the ordinary map and the red face degrees of the dessin d'enfant to match the boundary face degrees of the fully simple map.

Below, this function is described by interpreting both maps and dessins d'enfant in terms of triples of permutations, as discussed in \cref{sec:preliminaries}. However, such an algebraic proof is strongly motivated by a geometric intuition that is illustrated by the following example.

\begin{example}
The ordinary map on the left of \cref{fig:intuition} is not fully simple, since the central vertex is shared by boundary face 1 and boundary face 2. 
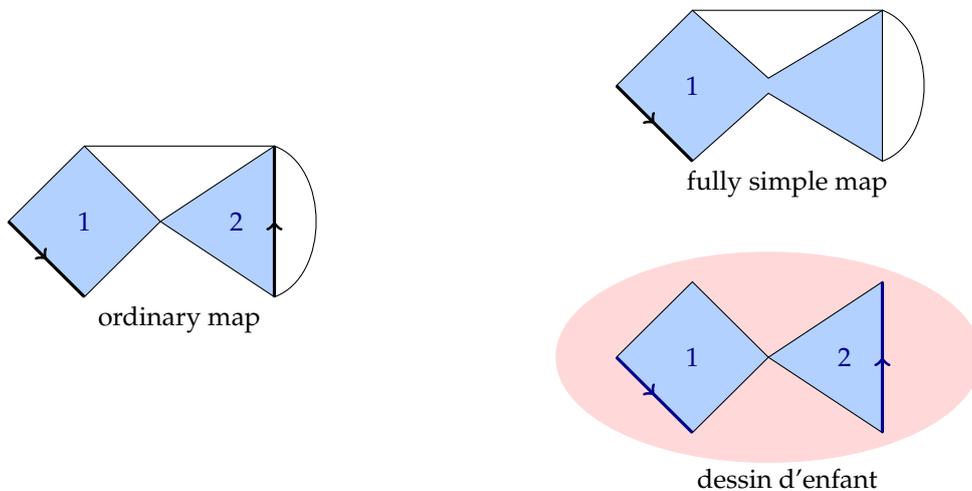
\begin{figure}[ht!]
\centering
\begin{tikzpicture}
\draw[fill = lightblue] (0,0) -- (1,1) -- (0,2) -- (-1,1) -- cycle;
\draw[fill = lightblue] (1,1) -- (2.5,0) -- (2.5,2) -- cycle;
\draw (2.5,0) to[out=20,in=-20] (2.5,2);
\draw(0,2) -- (2.5,2);
\node[darkblue] at (0,1) {1};
\node[darkblue] at (2,1) {2};
\node at (1.25,-0.3) {ordinary map};
\begin{scope}[very thick,decoration={markings,mark=at position 0.5 with {\arrow{>}}}] 
	\draw[postaction={decorate}] (-1,1) -- (0,0);
	\draw[postaction={decorate}] (2.5,0) -- (2.5,2);
\end{scope}
\begin{scope}[shift={(8,1.8)}]
\draw[fill = lightblue] (0,0) -- (1,0.9) -- (2.5,0) -- (2.5,2) -- (1,1.1) -- (0,2) -- (-1,1) -- cycle;
\draw (2.5,0) to[out=20,in=-20] (2.5,2);
\draw(0,2) -- (2.5,2);
\node[darkblue] at (0,1) {1};
\node at (1.25,-0.3) {fully simple map};
\begin{scope}[very thick,decoration={markings,mark=at position 0.5 with {\arrow{>}}}] 
	\draw[postaction={decorate}] (-1,1) -- (0,0);
\end{scope}
\end{scope}
\begin{scope}[shift={(8,-1.8)}]
\fill[fill = lightred] (1,1) ellipse (2.8 and 1.4);
\draw[fill = lightblue] (0,0) -- (1,1) -- (0,2) -- (-1,1) -- cycle;
\draw[fill = lightblue] (1,1) -- (2.5,0) -- (2.5,2) -- cycle;
\node[darkblue] at (0,1) {1};
\node[darkblue] at (2,1) {2};
\node at (1.25,-0.6) {dessin d'enfant};
\begin{scope}[very thick,decoration={markings,mark=at position 0.5 with {\arrow{>}}}] 
	\draw[darkblue, postaction={decorate}] (-1,1) -- (0,0);
	\draw[darkblue, postaction={decorate}] (2.5,0) -- (2.5,2);
\end{scope}
\end{scope}
\end{tikzpicture}
\caption{This example shows the geometric intuition that leads to the construction described below. }
\label{fig:intuition}
\end{figure}

By ``splitting'' the central vertex, one obtains the map on the right, which is indeed fully simple. In a certain sense, the dessin d'enfant below it stores the information required to recover the original map from the fully simple map, by gluing the blue faces of the dessin d'enfant into the fully simple map. So informally speaking, the fully simple map encodes the internal faces of the map while the dessin d'enfant encodes how the boundaries of the map intersect.
\end{example}

\subsection{Construction} \label{subsec:construction}

\subsubsection*{Forward}

We now describe our construction, which takes an ordinary map $M$ and returns a pair $(F(M), D(M))$, where $F(M)$ is a fully simple map and $D(M)$ is a dessin d'enfant. A consequence of the construction will be that the blue face degrees of $D(M)$ match the boundary face degrees of $M$ and the red face degrees of $D(M)$ match the boundary face degrees of $F(M)$.

{\em Input.} The ordinary map $M$ given by the data $(\sigma_0, \sigma_1, \sigma_2; R)$.

As in \cref{subsec:maps}, let $\mathbb{E}$ denote the set of oriented edges of $M$, so that $\sigma_0, \sigma_1, \sigma_2 \in \mathfrak{S}(\mathbb{E})$. Let $B \subseteq \mathbb{E}$ be the union of the $\sigma_2$-orbits of the elements of $R$. Consider the function $\partial:\mathfrak{S}(\mathbb{E}) \to \mathfrak{S}(B)$ that expresses a permutation on the set $\mathbb{E}$ as a union of disjoint cycles and then deletes those elements that do not lie in~$B$. Furthermore, let $\iota:\mathfrak{S}(B) \to \mathfrak{S}(\mathbb{E})$ be the natural inclusion. With this notation, the permutation $\sigma_0^\partial$ introduced at the end of \cref{subsec:maps} can be considered as an element of $\mathfrak{S}(\mathbb{E})$, namely $\iota \circ \partial(\sigma_0)$.

{\em Output.} The map $F(M)$ given by the triple of permutations $((\sigma_0^\partial)^{-1} \sigma_0, \sigma_1, \sigma_2 \sigma_0^\partial; \overline{R})$ and the dessin d'enfant $D(M)$ given by the triple of permutations $(\partial(\sigma_2 \sigma_0^\partial)^{-1}, \partial(\sigma_2), \partial(\sigma_0); R)$.

Notice that each (oriented) boundary edge in $M$ appears as an unoriented edge in $D(M)$, so that $B$ corresponds to the set of unoriented edges of $D(M)$.

To define $\overline{R}$, we first observe that there exists a set of faces of the unrooted map $F(M)$ whose adjacent edges precisely recover the set $B$. We designate these the boundary faces of $F(M)$ and suppose that their degrees are given by $\mu_1, \mu_2, \ldots, \mu_n$, in some order.

Next, consider the total order on the set $B$ of boundary edges, described in \cref{subsec:algorithm}. In other words, assign to a boundary edge the label $(i,j)$ if it is adjacent to face $i$ and it is equal to $\sigma_2^{j-1}(r_i)$, where we choose $j$ to be the smallest such positive integer. The total order is then simply the lexicographical order on $B$ with respect to these labels. Now choose the tuple of roots $\overline{R} = (\overline{r}_1, \overline{r}_2, \ldots, \overline{r}_n)$ such that $\overline{r}_i$ is adjacent to a face of degree $\mu_i$ and such that $(\overline{r}_1, \overline{r}_2, \ldots, \overline{r}_n)$ is lexicographically minimal.

\subsubsection*{Reverse}

We now describe the reverse construction, which takes a pair $(F, D)$ comprising a fully simple map and a dessin d'enfant whose boundary face degrees and red face degrees match and returns an ordinary map $M(F, D)$ whose boundary face degrees match the blue face degrees of $D$.

{\em Input.} The fully simple map $F$ given by the data $(\rho_0, \rho_1, \rho_2; \overline{R})$ and the dessin d'enfant $D$ given by the data $(\tau_r, \tau_b, \tau_v; R)$. We assume that the cycle type of $\partial(\rho_2)$ equals the cycle type of $\tau_r$.

We consider $\rho_0, \rho_1, \rho_2 \in \mathfrak{S}(\mathbb{E})$ and $\tau_r, \tau_b, \tau_v \in \mathfrak{S}(B)$. Suppose that the fully simple map $F$ has $n$ boundary faces of respective degrees $\mu_1, \mu_2, \ldots, \mu_n$. We assign to the oriented edges adjacent to boundary face $i$ the labels
\[
(i, 1), (i, 2), (i, 3), \ldots, (i, \mu_i),
\]
where $(i,1)$ denotes the root and the remaining labels are assigned in an anticlockwise manner around the boundary face. In a similar manner, suppose that the dessin d'enfant $D$ has $\ell$ blue faces of respective degrees $\lambda_1, \lambda_2, \ldots, \lambda_\ell$. We assign to the (unoriented) edges adjacent to blue face $i$ the labels
\[
(i, 1), (i, 2), (i, 3), \ldots, (i, \lambda_i),
\]
where $(i,1)$ denotes the root and the remaining labels are assigned in an anticlockwise manner around the blue face.

Thus, we have total orders on the set of boundary edges of the fully simple map $F$ and the set $B$ of edges of the dessin d'enfant $D$. Consider the unique order-preserving map between these two sets, which defines an embedding $B \to \mathbb{E}$ and hence, a natural inclusion $\iota: \mathfrak{S}(B) \to \mathfrak{S}(\mathbb{E})$.

{\em Output.} The map $M(F, D)$ given by the triple of permutations $(\iota(\tau_v) \rho_0, \rho_1, \rho_2 \iota(\tau_v)^{-1}; R)$.

\begin{example}
The ordinary map on the left of \cref{fig:construction} is not fully simple, since all of the vertices of boundary face 2 are not fully simple. Our construction produces a fully simple map with four connected components, which keeps track of the internal faces of the ordinary map. The dessin d'enfant instead keeps track of the original boundary faces. Observe that in this case, it has two connected components, since the two boundary faces of the ordinary map were disjoint. It is a general fact that the number of connected components of the dessin d'enfant equals the number of connected components of the boundary faces in the ordinary map.

\begin{figure}[ht!]
\centering
\def\svgwidth{1.1\columnwidth}
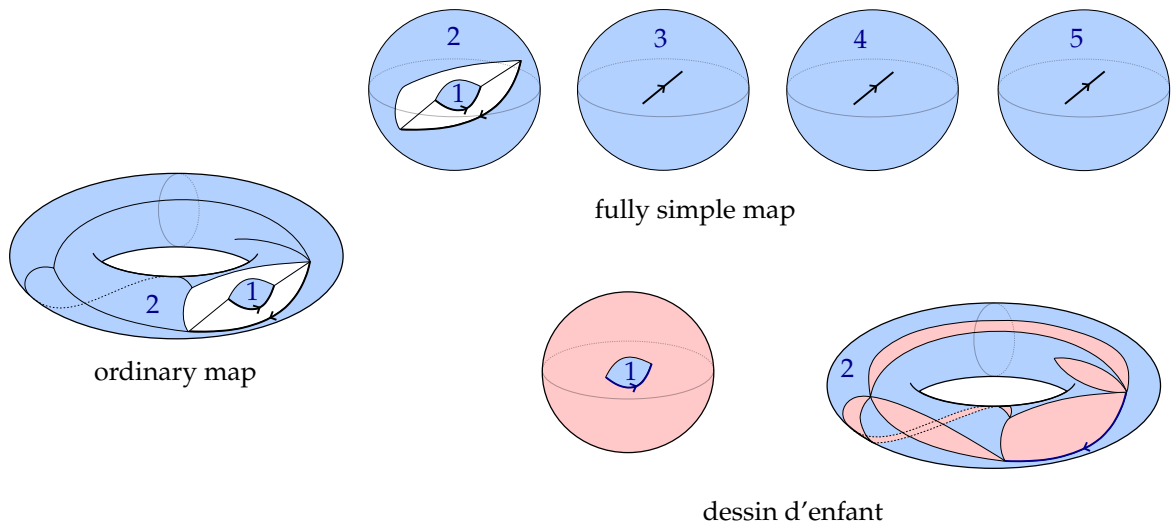
\caption{Illustration of the construction for the ordinary map on a torus from \cref{fig:map}.}
\label{fig:construction}
\end{figure}
\end{example}

\subsection{Proofs} \label{subsec:proofs}

We now check the various claims that were made regarding the construction of the previous section, before proving \cref{thm:main}.

{\bf The map $F(M)$ is a fully simple map.} \\
By construction, $F(M)$ is a map and it remains to check that it is fully simple. Using the characterisation of fully simple maps from \cref{subsec:maps}, this is equivalent to the fact that the elements of $B$ lie in different orbits of $(\sigma_0^\partial)^{-1} \sigma_0$, which is a direct consequence of the following result.

\begin{lemma}
Let $B \subseteq \mathbb{E}$ and $\sigma \in \mathfrak{S}(\mathbb{E})$. Define $\sigma^{\partial} = \iota \circ \partial(\sigma)$, where $\partial:\mathfrak{S}(\mathbb{E}) \to \mathfrak{S}(B)$ expresses a permutation on the set $\mathbb{E}$ as a union of disjoint cycles and then deletes those elements that do not lie in $B$, while $\iota:\mathfrak{S}(B) \to \mathfrak{S}(E)$ is the natural inclusion map. Then the elements of $B$ lie in different cycles of $(\sigma^{\partial})^{-1} \sigma$.
\end{lemma}

\begin{proof}
Suppose that $\sigma$ is written in disjoint cycle notation and let us describe how precomposing with $(\sigma^{\partial})^{-1}$ modifies its cycles. Observe that $\sigma^{\partial}(c)=c$, if $c\not\in B$. Therefore, if a cycle of $\sigma$ contains no element from $B$, then $(\sigma^{\partial})^{-1}\sigma$ contains the same cycle. 

If $\sigma(c)=b\in B$, then we choose the smallest positive integer $r$ such that $\sigma^{-r}(b)\in B$. By construction of $\sigma^{\partial}$, we have that $(\sigma^{\partial})^{-1}(b)=\sigma^{-r}(b)$. Then $\sigma^q\sigma^{-r}(b)\not \in B$ for $0<q<r$ and $(\sigma^{\partial})^{-1}$ leaves all of these elements invariant. In the cycle of $(\sigma^{\partial})^{-1}\sigma$ containing $\sigma^{-r}(b)$, the latter is followed by $\sigma^{-(r - 1)}(b),\ldots,\sigma^{-1}(b) = c$ which is then followed by $\sigma^{-r}(b)$. Therefore, it contains a unique element of $B$ and we have justified the claim. Note that if the initial cycle of $\sigma$ already contains only one $b$ belonging to $B$, then the minimum $r$ will be such that $\sigma^{-r}(b)=b$ and the resulting cycle of $(\sigma^{\partial})^{-1}\sigma$ will be the same as the initial one.
\end{proof}

{\bf $D(M)$ is a dessin d'enfant.} \\
We simply need to check that $\partial(\sigma_2 \sigma_0^\partial)^{-1} \circ \partial(\sigma_2) \circ \partial(\sigma_0) = \mathrm{id}$, which rearranges to $\partial(\sigma_2) \circ \partial(\sigma_0) = \partial(\sigma_2 \sigma_0^\partial)$. This is a direct consequence of the following lemma, under the identification $a = \sigma_2$ and $b = \partial(\sigma_0)$. We omit the proof, since it is a straightforward computation.

\begin{lemma}
Let $B \subseteq \mathbb{E}$, $a \in \mathfrak{S}(\mathbb{E})$ and $b \in \mathfrak{S}(B)$. If we define the maps $\partial$ and $\iota$ as in the previous lemma, then
\[
\partial(a) \circ b = \partial(a \circ \iota(b)).
\]
\end{lemma}

{\bf The boundary face degrees of $M$ match the blue face degrees of $D(M)$.} \\
This is true since both the boundary face degrees of $M$ and the blue face degrees of $D(M)$ are given by the cycle type of $\partial(\sigma_2)$.

{\bf The boundary face degrees of $F(M)$ match the red face degrees of $D(M)$.} \\
This is true since the boundary face degrees of $F(M)$ are given by the cycle type of $\partial(\sigma_2 \sigma_0^\partial)$, while the red face degrees of $D(M)$ are given by the cycle type of $\partial(\sigma_2 \sigma_0^\partial)^{-1}$. The cycle types agree since the permutations are inverses of each other.

\begin{proof}[Proof of \cref{thm:main}]
To an ordinary map $M$, the forward construction of \cref{subsec:construction} associates a fully simple map $F(M)$ as well as a dessin d'enfant $D(M)$. We have shown that the boundary face degrees of $M$ match the blue face degrees of $D(M)$ and that the boundary face degrees of $F(M)$ match the red face degrees of $D(M)$. Conversely, to a fully simple map $F$ and a dessin d'enfant $D$ satisfying these degree conditions, the reverse construction of \cref{subsec:construction} associates an ordinary map $M(F,D)$. These two constructions are precisely the inverses of each other, which can be shown by a straightforward check. Indeed, the triple of permutations describing the map $M(F, D)$ is obtained by inverting the relation between $M$ and the pair $(F(M),D(M))$.

The result then follows from the fact that the bijective map of the previous paragraph is weight-preserving. To see this, we first note that the construction preserves the internal faces and their degrees, when constructing $F(M)$ from $M$. So the powers of $t_1, t_2, t_3, \ldots$ agree. It then remains to check that the exponent of $\h$ is preserved as well. The exponent of $\h$ associated to the map $M = (\sigma_0, \sigma_1, \sigma_2; R)$ is
\[
-\chi(M) = -c(\sigma_0) + c(\sigma_1) - c(\sigma_2) + \ell(\mu),
\]
where $c(\sigma)$ denotes the number of disjoint cycles in the permutation $\sigma$. The exponent of $\h$ associated to the fully simple map $F(M)$ is
\[
-\chi(F(M)) = -c((\sigma_0^\partial)^{-1} \sigma_0) + c(\sigma_1) - c(\sigma_2 \sigma_0^\partial) + \ell(\lambda).
\]
It follows from \cref{lem:dessins} that the exponent of $\h$ associated to the dessin d'enfant $D(M)$ is
\[
|B| - c(\tau_v) = |B| - c(\partial(\sigma_0)).
\]
So we simply need to check that the first contribution is equal to the sum of the latter two. This can be expressed by the equation
\[
|B| - c(\partial(\sigma_0)) = c((\sigma_0^\partial)^{-1} \sigma_0) - c(\sigma_0),
\]
where we have used the fact that the internal faces of $M$ and $F(M)$ agree, which implies that $c(\sigma_2) - \ell(\mu) = c(\sigma_2 \sigma_0^\partial) - \ell(\lambda)$. However, this equation is a consequence of the two relations
\[
|B| = c(\partial((\sigma_0^\partial)^{-1} \sigma_0)) \qquad \text{and} \qquad c((\sigma_0^\partial)^{-1} \sigma_0) - c(\partial((\sigma_0^\partial)^{-1} \sigma_0)) = c(\sigma_0) - c(\partial(\sigma_0)).
\]
The first relation is equivalent to the obvious fact that in the fully simple map $F(M)$, the number of vertices incident to a boundary face is equal to the number of boundary edges. The second relation is equivalent to the obvious fact that the number of vertices in $F(M)$ that are not incident to a boundary face is equal to the number of vertices in $M$ that are not incident to a boundary face.
\end{proof}

\section{Generalisations} \label{sec:generalisations}

In this section, we generalise our result for two fundamental objects --- namely, stuffed maps and hypermaps. They can be roughly thought of as being akin to maps, with the same notion of boundaries but endowed with a richer internal structure. The idea for both generalisations is that the previous manipulations did not affect the internal structure of maps; thus, all the arguments still apply in these settings. We would like to point out that the same idea can be applied to more general objects, as long as they have the same type of boundaries as maps and the weight factorises as a product of contributions from boundary faces and internal faces.

\subsection{Stuffed maps}

In the definition of a map (\cref{def:map}), if one relaxes the condition that the complement of the graph is a disjoint union of topological disks, one obtains the notion of a {\em stuffed map} \cite{bor14}. However, we also impose the following important caveat --- namely, that the boundary faces must be homeomorphic to topological disks. On the other hand, the internal faces of a stuffed map may have arbitrary topology, including any non-negative genus and any positive number of boundary components. Each such boundary component then has an associated positive integer degree. The definition of fully simple (\cref{def:fullysimple}) then carries over verbatim to the context of stuffed maps. The enumeration of stuffed maps that we are concerned with requires extra parameters to keep track of the possible topologies of the internal faces.

\begin{definition}
For positive integers $\mu_1, \mu_2, \ldots, \mu_n$, let $\text{Map}^{\mathrm{st}}(\mu_1, \mu_2, \ldots, \mu_n)$ denote the weighted enumeration of stuffed maps such that the degree of boundary face $i$ is $\mu_i$ for $i = 1, 2, \ldots, n$. The weight of a stuffed map $M$ is given by
\[
\frac{\h^{-\chi(M)}}{|\mathrm{Aut}\,M|} \, \prod_{g \geq 0} \prod_\lambda t_{g,\lambda}^{f_{g,\lambda}(M)},
\]
where $\chi(M)$ is the Euler characteristic of the underlying surface with the interiors of the boundary faces removed and $f_{g,\lambda}(M)$ the number of internal faces with genus $g$ and $\ell(\lambda)$ boundary components with degrees prescribed by the non-empty partition~$\lambda$. Let $\text{FSMap}^{\mathrm{st}}(\mu_1, \mu_2, \ldots, \mu_n)$ denote the analogous weighted enumeration restricted to the set of fully simple stuffed maps.
\end{definition}

As in the usual case, the enumerations $\text{Map}^{\mathrm{st}}(\mu_1, \mu_2, \ldots, \mu_n)$ and $\text{FSMap}^{\mathrm{st}}(\mu_1, \mu_2, \ldots, \mu_n)$ are well-defined elements of $\mathbb{Z}[[\h, \h^{-1}; t_{g,\lambda} \mid g\geq 0 \text{ and } \lambda \text{ a partition}]]$.

Our main result extends to the enumerations of stuffed maps and fully simple stuffed maps, essentially without change.

\begin{theorem} \label{thm:stuffed}
For any partition $\lambda = (\lambda_1, \lambda_2, \ldots, \lambda_\ell)$ of a positive integer $d$,
\[
\mathrm{Map}^{\mathrm{st}}(\lambda) = z(\lambda) \sum_{\mu \vdash d} H^<(\lambda; \mu) \, \mathrm{FSMap}^{\mathrm{st}}(\mu).
\]
\end{theorem}

In order to extend our proofs to stuffed maps, one can consider a mild generalisation of the permutation model for maps presented in \cref{subsec:maps}. The basic idea is that a stuffed map can be encoded by a map, along with a partition of its internal faces and the assignment of a non-negative integer to each part in the partition. To recover the stuffed map from the map, for each part in the partition, we remove the corresponding internal faces and glue in a surface whose genus is specified by the associated integer.

\begin{lemma} 
A stuffed map can be encoded by a map $(\sigma_0, \sigma_1, \sigma_2; R)$ as per \cref{def:map}, along with an unordered partition $\mathcal P$ of the unrooted cycles of $\sigma_2$ and a function $h:{\mathcal P} \to \{0, 1, 2, \ldots\}$ that assigns a non-negative integer to each part of the partition.

The data $(\sigma_0, \sigma_1, \sigma_2; R; {\mathcal P}, h)$ and $(\widetilde{\sigma_0}, \widetilde{\sigma_1}, \widetilde{\sigma_2}; \widetilde{R}; \widetilde{\mathcal P}, \widetilde{h})$ define equivalent stuffed maps if and only if there exists an equivalence of maps $\phi: \mathbb{E} \to \widetilde{\mathbb{E}}$ such that $\phi$ carries $\mathcal P$ into $\widetilde{\mathcal P}$ and $h = \widetilde{h} \circ \phi$.
\end{lemma}

Note that, unlike for maps, the cycles of $\sigma_2$ do not necessarily correspond to faces. The faces of the stuffed map are rather the parts of $\mathcal{P}$. The Euler characteristic of a stuffed map is computed using this new notion of face.

As for maps, the boundary faces in stuffed maps must be homeomorphic to disks. Since all the manipulations in our proofs leave the internal faces unchanged and only affect boundary faces, the arguments of both \cref{sec:transpositions,sec:dessins} remain valid for stuffed maps and yield \cref{thm:stuffed}.

\subsection{Hypermaps}

Hypermaps generalise maps analogously to the way that hypergraphs generalise graphs. Whereas graphs and maps have edges that connect two vertices, hypergraphs and hypermaps have so-called hyperedges that connect any number of vertices. We will consider our hypermaps to have faces coloured blue and hyperedges coloured red. This leads to the following definition, adapted from \cite{lan-zvo04}.

\begin{definition}
A \emph{hypermap} is a bicoloured map, in the sense that each face is coloured either blue or red so that each edge is adjacent to a blue face and a red face. We furthermore require that each root is adjacent to a blue face. In this context, we consider the blue faces with roots as boundary faces, the blue faces without roots as internal faces, and the red faces as hyperedges.
\end{definition}

A dessin d'enfant is a particular case of a hypermap, in which there are no internal faces. As with the dessins d'enfant appearing in \cref{subsec:dessins}, we assign an orientation to the edges of a hypermap in which each edge is oriented to have a blue face on its left and a red face on its right. The definition of fully simple (\cref{def:fullysimple}) again carries over verbatim to the context of hypermaps.

One can encode an unrooted hypermap via a triple $(\sigma_0, \sigma_1, \sigma_2)$ of permutations acting on the set $E$ of (unoriented) edges, in which
\begin{itemize}
\item $\sigma_0$ rotates each edge anticlockwise by two edges around the vertex to which it points;
\item $\sigma_1$ rotates each edge anticlockwise around the adjacent red face; and
\item $\sigma_2$ rotates each edge anticlockwise around the adjacent blue face.
\end{itemize}
It follows that $\sigma_0 \sigma_1 \sigma_2 = \mathrm{id}$.

\begin{lemma}
A hypermap can be encoded by a triple $(\sigma_0, \sigma_1, \sigma_2)$ of permutations in $\mathfrak{S}(E)$ and a tuple $R \in E^n$ such that
\begin{itemize}
\item $\sigma_0 \sigma_1 \sigma_2 = \mathrm{id}$; and
\item no two elements of $R$ lie in the same cycle of $\sigma_2$.
\end{itemize}
The data $(\sigma_0, \sigma_1, \sigma_2; R)$ and $(\widetilde{\sigma}_0, \widetilde{\sigma}_1, \widetilde{\sigma}_2; \widetilde{R})$ define equivalent hypermaps if and only if there exists a bijection $\phi: E \to \widetilde{E}$ that sends $R$ to $\widetilde{R}$ and satisfies $\widetilde{\sigma}_i = \phi \sigma_i \phi^{-1}$ for $i \in \{0, 1, 2\}$.
\end{lemma}

A dessin d'enfant given by the triple $(\tau_r, \tau_b, \tau_v)$ can be thought of as a hypermap given by the triple $(\sigma_0, \sigma_1, \sigma_2)$ under the correspondence $\tau_r = \sigma_1, \tau_b = \sigma_2, \tau_v = \sigma_0$. Notice that the equation $\tau_r\tau_b\tau_v = \mathrm{id}$ is then equivalent to $\sigma_0\sigma_1\sigma_2 = \mathrm{id}$. Also observe that if $\sigma_1$ (or $\tau_r$) is a fixed point free involution, then the hyperedges (or red faces) have degree 2 and can be collapsed to become edges in the usual sense. In this case, the hypermap is a map.

\begin{definition}
For positive integers $\mu_1, \mu_2, \ldots, \mu_n$, let $\text{Map}^{\mathrm{h}}(\mu_1, \mu_2, \ldots, \mu_n)$ denote the weighted enumeration of hypermaps such that the degree of boundary face $i$ is $\mu_i$ for $i = 1, 2, \ldots, n$. The weight of a map $M$ is given by
\[
\frac{\h^{-\chi(M)}}{|\mathrm{Aut}\,M|} \, t_{1}^{f_1(M)}t_{2}^{f_2(M)}t_3^{f_3(M)}\cdots u_1^{e_1(M)} u_2^{e_2(M)} u_3^{e_3(M)}\cdots,
\]
where $f_i(M)$ is the number of internal faces of degree $i$ and $e_i(M)$ is the number of hyperedges of degree~$i$. Let $\text{FSMap}^{\mathrm{h}}(\mu_1, \mu_2, \ldots, \mu_n)$ denote the analogous weighted enumeration restricted to the set of fully simple hypermaps.
\end{definition}

As in the usual case, the enumerations $\text{Map}^{\mathrm{h}}(\mu_1, \mu_2, \ldots, \mu_n)$ and $\text{FSMap}^{\mathrm{h}}(\mu_1, \mu_2, \ldots, \mu_n)$ are well-defined elements of $\mathbb{Z}[[\h, \h^{-1}; t_1, t_2, t_3, \ldots; u_1, u_2, u_3, \ldots]]$.

Our main result extends to the enumerations of hypermaps and fully simple hypermaps, essentially without change.

\begin{theorem} \label{thm:hypermaps}
For any partition $\lambda = (\lambda_1, \lambda_2, \ldots, \lambda_\ell)$ of a positive integer $d$,
\[
\mathrm{Map}^{\mathrm{h}}(\lambda) = z(\lambda) \sum_{\mu \vdash d} H^<(\lambda; \mu) \, \mathrm{FSMap}^{\mathrm{h}}(\mu).
\]
\end{theorem}

Again, this result follows from our arguments in \cref{sec:transpositions,sec:dessins}, since the boundary faces in hypermaps must be homeomorphic to disks and all the manipulations in our proofs leave the internal faces unchanged and only affect boundary faces.

\bibliographystyle{amsplain}
\bibliography{fully-simple-maps.bib}

\end{document}